\documentclass[12pt]{amsart}
\usepackage{amssymb}
\usepackage{enumerate,color}
\usepackage{graphicx}
\usepackage[text={6in,9in},centering]{geometry}
\usepackage{subfigure}

\theoremstyle{plain}
\newtheorem{thm}{Theorem}[section]
\newtheorem{cor}[thm]{Corollary}
\newtheorem{lem}[thm]{Lemma}
\newtheorem{prop}[thm]{Proposition}
\theoremstyle{definition}
\newtheorem{defn}{Definition}
\newtheorem{exmp}{Example}
\theoremstyle{remark}

\def\R{\mathbb{R}}

\newcommand{\bu}{\mathbf{u}}

\newcommand{\be}{\begin{equation}} 
	\newcommand{\ee}{\end{equation}} 

\newcommand\asd{\mbox{\rm dim}_{\rm A}\,} 
\newcommand\pkd{\mbox{\rm dim}_{\rm P}\,} 
\newcommand\hdd{\mbox{\rm dim}_{\rm H}\,} 
\newcommand\ubd{\overline{\mbox{\rm dim}}_{\rm B}\,} 
\newcommand\lbd{\underline{\mbox{\rm dim}}_{\rm B}\,} 
\newcommand\bod{\mbox{\rm dim}_{\rm B}\,} 
\newcommand\uid{\overline{\mbox{\rm dim}}_{\rm \theta}\,} 
\newcommand\lid{\underline{\mbox{\rm dim}}_{\rm \theta}\,} 

\title{Intermediate dimensions of  Moran sets and their visualization}
\date{}
\author{Yali Du}
\address[Yali Du]{School of Mathematics and Statistics, Hubei University, Wuhan, 430062, P.~R. China}
\email{yl.du@yahoo.com}

\author[J. J. Miao]{Jun Jie Miao}
\address[J. J. Miao]{School of Mathematical Sciences,  Key Laboratory of MEA(Ministry of Education) \& Shanghai Key Laboratory of PMMP,  East China Normal University, Shanghai 200241, P.~R. China}
\email{jjmiao@math.ecnu.edu.cn}

\author[Tianrui Wang]{Tianrui Wang}
\address[Tianrui Wang]{School of Mathematical Sciences,  Key Laboratory of MEA(Ministry of Education) \& Shanghai Key Laboratory of PMMP,  East China Normal University, Shanghai 200241, P.~R. China}
\email{51265500036@stu.ecnu.edu.cn}

\author[Haojie Xu]{Haojie Xu}
\address[Haojie Xu]{School of Mathematical Sciences,  Key Laboratory of MEA(Ministry of Education) \& Shanghai Key Laboratory of PMMP,  East China Normal University, Shanghai 200241, P.~R. China}
\email{51265500075@stu.ecnu.edu.cn}
\keywords{Intermediate dimension, Hausdorff dimension, Box-counting dimension, Moran set, Cut set}
\thanks{This work is partially supported by National Natural Science Foundation of China (Grant No.12201190).}

\begin{document}
	\begin{abstract}
		Intermediate dimensions are a class of new fractal dimensions which provide a spectrum of dimensions interpolating between the Hausdorff and box-counting dimensions. 
	
		In this paper, we study the intermediate dimensions of Moran sets. Moran sets may be regarded as a generalization of self-similar sets generated by using different class of similar mappings at each level with unfixed translations, and this  causes the lack of ergodic properties on Moran set. Therefore, the intermediate dimensions do not necessarily exist, and we calculate  the upper and lower intermediate dimensions of Moran sets. In particular, we obtain a simplified intermediate dimension formula for homogeneous Moran sets. Moreover, we study the visualization of  the upper intermediate dimensions for some homogeneous Moran sets, and  we show that  their upper intermediate dimensions are given by M\"{o}bius transformations.
	\end{abstract}
	
	\maketitle
	
	\section{Introduction}
	\subsection{Intermediate dimensions}
	The notion of dimension is central to fractal geometry, and there are different dimensions used in studying the various fractal objects such as  lower dimension, Hausdorff dimension, packing dimension, box-counting dimension and Assouad dimension, see~\cite{Fal,Fra2,Wen00}. It is well know that all these dimensions are identical for self-similar sets satisfying open set condition. In various studies,  Hausdorff and box-counting dimensions are  two fundamental ones used in fractal geometry, and there are many interesting fractal sets with different Hausdorff and box-counting dimensions. For example, the Hausdorff dimensions of many non-typical self-affine carpets and Moran sets are strictly less than their box-counting dimensions, see  \cite{Baran07,Bedfo84, LalGa92,McMul84, Wen00,Wen01}. The reason is because covering sets of widely ranging scales are permitted in the definition of Hausdorff dimensions, whereas  covering sets that are all of the same size are essentially used in box-counting dimensions, see~\cite{Fal} for details.

Recently, the growing literature on dimension spectra is starting to provide a unifying framework for the many notions of dimensions that arise throughout the field of fractal geometry, see  \cite{Ban1,Fra1, FFK,FHHTY,FY} for various studies on dimension spectra.  Suppose that there are two different dimensions, written as $\dim_X$ and $\dim_Y$, with $\dim_X E \leq \dim_YE$ for all $E \subset \mathbb{R}^d$. Dimension spectra aim to provide a continuum of dimensions, say $\dim_{\theta}$ with $\theta \in [0,1]$, such that
	$$
	\dim_0 E = \dim_X E \qquad \textit{and }\qquad \dim_1 E = \dim_Y E.
	$$
This is of interest for many reasons. For example, Hausdorff and box-counting dimensions 	may behave differently for many non-typical self-affine sets and Moran sets since each of them is sensitive to different geometric properties of these sets. Therefore, it may be valuable to understand for which $\theta$ this transition in geometric behaviour occurs, and this potentially 	deepens our understanding of Hausdorff  and box-counting dimensions and the geometric structure of the fractal sets, see \cite{BJ1, BJ2,BanKol,B, BFF,Fra21, Fra4,Kol1} for various related studies and applications.
	
	Recently, Falconer, Fraser and Kempton in \cite{FFK} introduced intermediate dimensions to provide a unifying framework for Hausdorff and box-counting dimensions. 	
	\begin{defn}
		Given a subset $E\subset \mathbb{R}^d$. For each $0\le \theta \le 1$, the lower and upper $\theta$-intermediate dimensions of $E$ are defined respectively by
		\begin{eqnarray*}
			\lid E&=&\inf\{  s\ge 0:\textit{for all }\varepsilon>0, \Delta>0, \textit{there exists } 0<\delta\le\Delta \textit{ and a cover $\{U_i\}$} \\
			&&\hspace{2cm} \textit{of $E$ such that } \delta^\frac{1}{\theta}\le \lvert U_i\rvert \le \delta , \sum \lvert U_i\rvert ^s\le \varepsilon\}, \\
			\uid E&=&\inf\{  s\ge 0:\textit{for all }\varepsilon>0,  \textit{ there exists } \Delta >0, \textit{such that for all }0<\delta\le\Delta \\
			&&\hspace{1cm}    \textit{ there is a cover $\{U_i\}$} \textit{of $E$ such that } \delta^\frac{1}{\theta}\le \lvert U_i\rvert \le \delta , \sum \lvert U_i\rvert ^s\le \varepsilon\}.
		\end{eqnarray*}
		If $\lid E=\uid E $, we write $\dim_{\theta}E$ for the common value which we refer to as the $\theta$-intermediate dimension of $E$.
	\end{defn}
	
	As we may see from the definition, intermediate dimensions provide a continuum between Hausdorff and box-counting dimensions since it is achieved by restricting the families of allowable covers in the definition of Hausdorff dimension by requiring  $|U| \leq |V |^{\theta}$ for all sets $U, V$ in an admissible cover, where $\theta \in [0, 1]$ is a parameter. For $\theta = 1$, the only covers using sets of the same size are allowable, and box-counting dimension is recovered. On the other hand, for $\theta = 0$, there are no restrictions for the size of the sets used in the covers, and this gives Hausdorff dimension. Therefore, Hausdorff and box-counting dimensions may be regarded as particular cases of a spectrum of intermediate dimensions $\dim_{\theta} E$, that is, 	
	$$
	\overline{\dim}_0 E =\underline{\dim}_0E=\hdd E, \quad  \overline{\dim}_1E=\overline{\dim}_B E, \quad   \underline{\dim}_1E=\underline{\dim}_B E,
	$$
and	 we refer the readers to \cite{FFK,Fra1} for the properties of intermediate dimensions. In \cite{FFK,Fra1}, the authors proved the continuity of intermediate dimensions.
	\begin{prop}
		Given a bounded set  $E \subset \R^d$, the dimension spectra $\lid E$ and $\uid E$ are  continuous functions for  $\theta\in (0,1]$.
	\end{prop}
Note that dimension spectra $\lid E$ and $\uid E$ are not necessarily continuous at $\theta =0$, see Example \ref{exm1} in Section~\ref{sec_VEID}.

Since intermediate dimensions provide a continuum between Hausdorff and box-counting dimensions, it is natural to investigate the dimensions spectra for the fractals sets with different Hausdorff and box-counting dimensions. In \cite{BanKol}, Banaji and Kolossvary studied intermediate dimensions for a class of non-typical self-affine sets, named Bedford-McMullen carpets, and they determined a precise formula for the intermediate dimensions of  Bedford–McMullen carpets for the whole spectrum of $\theta \in [0,1]$. In \cite{B,BFF}, Burrell, Falconer and Fraser show that the intermediate dimensions of the projection of a set $E\in \R^d $ by “intermediate dimension profiles”.

	\subsection{Moran sets}	
	
	Since Moran sets are a class of important fractal sets, they are frequently used as a testing ground on questions and conjectures of fractal, see \cite{Wen01}. In this paper, we  investigate the properties of intermediate dimensions for Moran sets.

	Let $\{n_{k}\}_{k\geq 1}$ be a sequence of integers greater than or equal to $2.$ For
	each $k=1,2,\cdots$, we write
	$$
	\Sigma^{k}=\{u_{1}u_{2}\cdots u_{k}:1\leq u_{j}\leq
	n_{j},\  j\leq k\}  \qquad  \textit{and } \qquad\Sigma^{*}=\bigcup _{k=0}^{\infty }\Sigma^{k}
	$$
	for the set of words of length $k$ and for the set of all finite words, respectively,  with $
	\Sigma^{0}=\{\emptyset \}$ containing only the empty word $\emptyset$.  We write
	$$
	\Sigma^{\infty}=\{\mathbf{u}=u_{1}u_{2}\cdots u_{k}\cdots :1\leq u_{k}\leq
	n_{k},\  k=1,2,\cdots \}
	$$
	for the set of words with infinity length, and we  topologize $\Sigma^\infty$ by using the metric
	$d(\mathbf{u},\mathbf{v})=2^{-|\mathbf{u}\wedge  \mathbf{v}|}$ for distinct $\mathbf{u},\mathbf{v} \in \Sigma^\infty$ to make $\Sigma^\infty$ into a compact metric space.
	For each $\mathbf{u}=u_1 \dots u_k \in \Sigma^*$, we write $\mathbf{u}^* =u_1 \dots u_{k-1}$.  Given $\mathbf{u}\in \Sigma^l$, for $\mathbf{v}\in \Sigma^k$ where $k\geq l$ or $\mathbf{v}\in\Sigma^\infty $,  we write $\mathbf{u}\prec \mathbf{v}$ if $u_i =v_i$ for all $i=1,2,\ldots l$.
	
	We define the \textit{cylinders} $\mathcal{C}_\mathbf{u}=\{\mathbf{v}\in \Sigma^\infty : \mathbf{u}\prec \mathbf{v}\}$ for $\mathbf{u}\in \Sigma^*$; the set of cylinders $\{\mathcal{C}_\mathbf{u} : \mathbf{u} \in \Sigma^* \}$ forms a base of open and closed neighborhoods for $\Sigma^\infty$. We term a subset $A$ of $\Sigma^*$ a
	\textit{cut set} if $\Sigma^\infty\subset\bigcup_{\mathbf{u}\in
		A}\mathcal{C}_\mathbf{u}$, where
	$\mathcal{C}_\mathbf{u}\bigcap\mathcal{C}_{\mathbf{v}}=\emptyset$ for  all $\mathbf{u}\neq
	\mathbf{v}\in A$. It is equivalent to that, for every $\mathbf{w}\in \Sigma^\infty$, there is a
	unique sequence $\mathbf{u}\in A$ with $|\mathbf{u}|<\infty$ such that $\mathbf{u}\prec \mathbf{w}$.

	Suppose that $
	J\subset \mathbb{R}^{d}$ is a compact set with $\mbox{int}(J)\neq
	\varnothing $ (we always write int($\cdot )$ for the interior of a set). Let $
	\{\phi _{k}\}_{k\geq 1}$ be a sequence of positive real vectors where $\phi
	_{k}=(c_{k,1},c_{k,2},\cdots, c_{k,n_{k}})$ and $\Sigma
	_{j=1}^{n_{k}}(c_{k,j})^{d}\leq 1$ for every integer $k>0$. 
	We say the
	collection $\mathcal{F}=\{J_{\mathbf{u}}: \mathbf{u}\in \Sigma^*\}$ of closed
	subsets of $J$ fulfills the \textit{Moran structure} if it satisfies the
	following Moran structure conditions (MSC):
	\begin{itemize}
		\item[(1).] For each $\mathbf{u}\in \Sigma^*$, $J_{\mathbf{u}}$ is geometrically similar to $J$, i.e., there exists a similarity $\Psi_{\mathbf{u}}:\mathbb{R}	^{d}\rightarrow \mathbb{R}^{d}$ such that $J_{\mathbf{u}}=\Psi_{\mathbf{u}}(J)$. We write $J_{\emptyset }=J$ for empty word $\emptyset $.
		\item[(2).] For all $k\in \mathbb{N}$ and $\mathbf{u}\in \Sigma^{k-1}$, the elements $J_{\mathbf{u}1}, J_{\mathbf{u}2},\cdots ,J_{\mathbf{u}{n_{k}}}$ of $\mathcal{F}$
		are the subsets of $J_{\mathbf{u}}$ with disjoint interiors, i.e., $\mbox{int}(J_{\mathbf{u}i})\cap \mbox{int}(J_{\mathbf{u}i^{\prime }})=\varnothing $
		for $i\neq i^{\prime }$. Moreover, for all $1\leq i\leq n_{k} $,
		\begin{equation*}
			\frac{|J_{\mathbf{u}i}|}{|J_{\mathbf{u}}|}=c_{k,i},
		\end{equation*}
		where $|\cdot |$ denotes the diameter of a set.
	\end{itemize}
	The non-empty compact set
	\begin{equation}\label{attractor}
		E=E(\mathcal{F})=\bigcap\nolimits_{k=1}^{\infty }\bigcup\nolimits_{\mathbf{u}\in \Sigma^{k}}J_{\mathbf{u}}
	\end{equation}
	is called a \textit{Moran set} determined by $\mathcal{F}$. In particular, if for each integer $k\geq 1$,  all entries of the vector $\phi_{k}=(c_{k,1},c_{k,2},\cdots ,c_{k,n_{k}})$  are identical, that is
	$$c_{k,i}=c_k, $$
for every  $i=1,2,\ldots , n_k$,	we call $E$ is a \textit{homogeneous Moran set}.  For all $\mathbf{u}\in \Sigma^{k}$, the elements $J_{\mathbf{u}}$ are called \textit{\ $k$th-level basic sets} of $E$.

For all $k^{\prime }>k\geq 0,$ let $s_{k,k^{\prime }}$ be the unique real solution
	of the equation $\Delta _{k,k^{\prime }}(s)=1$, where
	\begin{equation}  \label{eqns}
		\Delta _{k,k^{\prime }}(s)=\prod\nolimits_{i=k+1}^{k^{\prime }}\left(
		\sum\nolimits_{j=1}^{n_{i}}(c_{i,j})^{s}\right) .
	\end{equation}
	For simplicity, we often write
	\begin{equation}\label{s1}
		s_k=s_{0,k}.
	\end{equation}
	Let $s_{\ast }$, $s^{\ast }$ and $s^{\ast \ast }$  be  the real numbers given respectively  by
	\begin{equation}\label{def_s***}
		s_{\ast }=\liminf_{m\rightarrow \infty }s_{m}, \quad
		s^{\ast }=\limsup_{m\rightarrow \infty }s_{m}, \quad  s^{\ast \ast }=\lim_{m\rightarrow \infty }\left(\sup\nolimits_{k} s_{k,k+m}\right).
	\end{equation}
	We write
	\begin{equation*}
		c_*=\inf_{k,j} c_{k,j}.
	\end{equation*}
	It was shown in \cite{HL96,LLMX, Wen00,Wen01} that if $c_{*}>0,$ then
	\begin{equation*}
		\hdd E=s_{\ast }, \quad \pkd E=\ubd E=s^{\ast }, \quad\asd E=s^{\ast \ast}.
	\end{equation*}

	The dimension theory of Moran sets has been studied extensively, and we refer the readers to \cite{HL96,Wen00,Wen01} for details and references therein. Note that, in the definition	of Moran sets, the position of $J_{\mathbf{u}i}$ in $J_\mathbf{u}$ is very flexible, and the contraction ratios	may also vary at each level. Therefore the structures of Moran sets are more complex than self-similar sets, and in general, the inequality
	$$\hdd E \leq \lbd E \leq \ubd E$$
	holds strictly for Moran fractals. The general lower box dimension formula for Moran	sets is still an open question. Except providing various examples, Moran sets are also useful tools for analysing properties of fractal sets in various studies, for example,	see [28] and references therein for applications.

	\subsection{Main conclusions}
	To study the intermediate dimensions, we have to analyse the covers of Moran sets.	Given $\delta>0$ and $\theta \in (0,1]$, we write
	\begin{eqnarray}\label{def_stheta}
		s_{\delta , \theta}&=&\min\big\{s:\sum_{\mathbf{u}\in \mathcal{M}} |J_\mathbf{u}|^s =1 \textit{ where $\mathcal{M}$ is a cut set such that}  \\
		&&\hspace{2cm}\delta^{\frac{1}{\theta}} < |J_{\mathbf{u}^*}| \textit{ and }  |J_\mathbf{u}|\le \delta \textit { for each $\mathbf{u} \in \mathcal{M}$ }\big\}.  \nonumber
	\end{eqnarray}
	Let $s^\theta$ and $s_\theta$  be  the upper and lower limits of $s_{\delta , \theta}$, respectively, that is
	\begin{equation}\label{s^*}
		s^\theta =\limsup_{\delta \rightarrow 0 } s_{\delta , \theta}, \qquad
		s_\theta=\liminf_{\delta \rightarrow 0 }s_{\delta , \theta}.
	\end{equation}
For $\theta=0$, we set
	$$
	s^\theta =s_\theta =s_*.
	$$
	
	Since geometric structure of Moran sets varies considerably between $c_*>0$ and $c_*=0$, we first state our conclusion for the Moran sets with $c_*>0$.
	\begin{thm}\label{thm_1}
		Let $E$ be a Moran set given by \eqref{attractor} with $c_*>0$. Then  the upper and lower  $\theta$-intermediate dimensions are given by
		$$
		\uid E=s^\theta, \qquad
		\lid E=s_\theta,
		$$
		where $s^\theta$ and $s_\theta$ are  given by \eqref{s^*}.
	\end{thm}

Unlike the self-similar sets, the above theorem implies that  the intermediate dimension of Moran sets does not necessarily exist unless that $s^\theta=s_\theta$. Furthermore,  the upper intermediate dimension of Moran sets often has more complex behaviours,  see  Example \ref{exm1} and Example \ref{exm2} in Section \ref{sec_VEID}.

	Let $E$ be a homogeneous Moran set, that is, for every  $k\geq 1$,  we have that $c_{k,i}=c_k$ for $i=1,2,\ldots, n_k$. For each integer $k\geq 1$, there exists a unique integer $l(k,\theta)=l$ such that
	\begin{equation}\label{def_lk}
		c_1c_2\ldots c_l\leq (c_1c_2\ldots c_k)^{\frac{1}{\theta}}<c_1c_2\ldots c_{l-1}.
	\end{equation}
	The upper and lower intermediate dimensions of homogeneous Moran sets have the following simplified forms.
	\begin{cor}\label{cor_HMS}
		Let $E$ be a homogeneous Moran set with $c_*>0$. Then
\begin{eqnarray*}
 \uid E &=&\limsup_{k\to\infty} \min_{k\leq m\leq l(k,\theta)} -\frac{\log n_1 \dots n_m}{\log c_1 \dots c_m}, \\
 \lid E &=&\liminf_{k\to\infty}  -\frac{\log n_1 \dots n_k}{\log c_1 \dots c_k}=\hdd E, 
\end{eqnarray*}
		where  $l(k,\theta)$ is given by  \eqref{def_lk}.
	\end{cor}
	
	The dimension formulas of Moran sets with $c_*=0$ are much more difficult to compute since the contraction  ratios in the vectors $\phi_k$ may decrease to $0$ extremely fast as $k$ tends to $\infty$.
	Therefore, we use the following  terms to control the decay speed of  $\phi_k$,
	$$
	\underline{c}_k=\min_{1\le j \le n_k} \{c_{k,j}\}, \quad  \textit{and }\quad M_k=\max_{\mathbf{u} \in \Sigma^k} \lvert J_\mathbf{u} \rvert,
	$$
see Example \ref{exm3} in Section \ref{sec_VEID}.	Under an extra assumption, we obtain the intermediate dimensions for Moran sets with $c_*=0$.
	\begin{thm}\label{thm_3}
		Let $E$ be a Moran set given by \eqref{attractor} with $c_*=0$. Suppose that
		$$
		\lim_{k\to +\infty} \frac{\log{\underline{c}_k}}{\log{M_k}}=0.
		$$
		Then the upper and lower intermediate dimensions are given by
		$$
		\uid E=s^\theta, \qquad
		\lid E=s_\theta,
		$$
		where $s^\theta$ and $s_\theta$ are given by \eqref{s^*}.
	\end{thm}
Similarly, we have the following special conclusion for homogenous Moran sets with $c_*=0$.	
	\begin{cor}\label{thm_4}
		Let $E$ be a homogeneous Moran set with $c_*=0$. Suppose that
		$$
		\lim_{k\to +\infty} \frac{\log{c_k}}{\log{c_1 \dots c_k}}=0.
		$$
		Then
\begin{eqnarray*}
 \uid E &=&\limsup_{k\to\infty} \min_{k\leq m\leq l(k,\theta)} -\frac{\log n_1 \dots n_m}{\log c_1 \dots c_m}, \\
 \lid E &=&\liminf_{k\to\infty}  -\frac{\log n_1 \dots n_k}{\log c_1 \dots c_k}=\hdd E,
\end{eqnarray*}
		where  $l(k,\theta)$ is given by  \eqref{def_lk}.
	\end{cor}
	
As we may notice that  all the formulas of intermediate dimensions are implicit functions of $\theta$, it is interesting to give intermediate dimensions in the explicit form of $\theta$. Given an integer $L\ge 2$, we write
	$$\mathcal{F}_L=\Big \{f(\theta)=\frac{La \theta +b}{Lc \theta +c} : a,b,c\in\mathbb{R}, c>a\ge b>0, \frac{a}{b}\in\mathbb{N} \Big\}.$$
In the final conclusion, we show that M\"{o}bius transformations may be used for the upper intermediate dimensions of some homogeneous Moran sets.	
	\begin{prop}\label{Prop_MF}
		Given an integer $L\ge 2$. For every $f\in \mathcal{F}_L$,  there exists a homogeneous Moran set $E$ such that
$$
\uid E=\left\{\begin{array}{lc} f(\theta),  & \textit{  for } \theta \in[\frac{1}{L^2},1];  \\
\hdd E,  & \textit{  for } \theta \in[0,\frac{1}{ L^2}],
\end{array} \right.
$$
 and $\lid E=\hdd E$ for  $\theta\in [0,1]$.
	\end{prop}

	\section{Intermediate dimension of Moran sets with $c_*>0$}
	In this section, we study the intermediate dimension of Moran sets with $c_*>0$. First, we state a conclusion for Moran sets regardless of $c_*$, and it is also applicable to Moran sets with $c_*=0$ in the next section.
	\begin{lem}\label{delta cover}
		Given a Moran set $E$,  a real $\delta>0$ and $\theta\in (0,1]$. Let $\mathcal{M}$ be a cut set of $\Sigma^\infty$ such that $\delta^{\frac{1}{\theta}} < |J_{\mathbf{u}^*}| $  and $ |J_\mathbf{u}|\le \delta$  for every $\mathbf{u} \in \mathcal{M} $. Then there exists a cover $\mathcal{F}_{\mathcal{M}}=\{U_\bu: \bu\in \mathcal{M}\}$ of $E$ such that $J_\mathbf{u} \subset U_\mathbf{u}$, $\delta^{\frac{1}{\theta}}\leq |U_\mathbf{u}|\le\delta$ and $|J_\mathbf{u}|\leq |U_\mathbf{u}|<|J_{\mathbf{u}^*}|$ for all $\bu\in \mathcal{M}$.
	\end{lem}	
	
	\begin{proof}
Since $ \mathcal{M}$ is a cut set satisfying that  $ \delta^{\frac{1}{\theta}} < |J_{\mathbf{u}^*}| $  and $ |J_\mathbf{u}|\le \delta$ for all $\mathbf{u} \in \mathcal{M}$. We define a cover  $\mathcal{F}_\mathcal{M}=\{U_{\mathbf{u}}: \mathbf{u}\in \mathcal{M}\}$ of $E$ by setting
		$$
		U_\mathbf{u}=\left\{ \begin{array}{ll}
		J_\mathbf{u}& \textit{ if } |J_\mathbf{u}|\ge\delta^{\frac{1}{\theta}}, \\
		\bigcup_{x\in J_\mathbf{u}} B(x,\frac{\delta^\frac{1}{\theta}- |J_{\mathbf{u}}|}{2}) & \textit{ if } |J_\mathbf{u}|<\delta^{\frac{1}{\theta}},
		\end{array} \right.
		$$
for every $\mathbf{u}\in\mathcal{M}$.

It is clear that  $J_\mathbf{u} \subset U_\mathbf{u}$, $\delta^{\frac{1}{\theta}}\leq |U_\mathbf{u}|\le\delta$ and $|J_\mathbf{u}|\leq |U_\mathbf{u}|<|J_{\mathbf{u}^*}|$   for each $\mathbf{u}\in\mathcal{M}$, and the conclusion holds.
	\end{proof}

Given a Moran set $E$. For sufficiently small $\delta$, we write
	\begin{equation}\label{CSM}
	\mathcal{M}(\delta) =\{\mathbf{u}\in\Sigma^*: |J_\mathbf{u}|\le\delta<|J_{\mathbf{u}^*}| \},
	\end{equation}
	and it is clear that $\mathcal{M}(\delta) $ is a cut set of $\Sigma^\infty$.
	For  $F\subset \mathbb{R}^d$ such that  $E\cap F\ne \emptyset$, we write
\begin{equation}\label{def_Aset}
	A(F)=\{\mathbf{u}:\mathbf{u}\in \mathcal{M}(|F|) , J_\mathbf{u}\cap F \ne \emptyset\}.
\end{equation}
	The following conclusion shows that the number of basic sets of $E$ with the similar size of $F$ is bounded and  independent of $F$.
	\begin{lem}\label{finite intersection}
		Let $E$ be a Moran set given by \eqref{attractor} with $c_*>0$.  Then there exists  a constant $C$ such that for  every  $F \subset \mathbb{R}^d$ such that  $E\cap F\ne \emptyset$,
		$$
		\#A(F)\le C.
		$$
	\end{lem}
	\begin{proof}
		Given  $F \subset \mathbb{R}^d$ such that  $E\cap F\ne \emptyset$. For every $\bu\in\mathcal{M}(|F|)$, we have that
		$$
 | J_{\mathbf{u}}|\geq c_* | J_{\mathbf{u}^*} | \ge 	c_* |F| ,
		$$
		For $x\in F$, since $J_\mathbf{u} \subset B(x,2\delta)$ for each $ \mathbf{u} \in A(F)$,  it follows that
		\begin{eqnarray*}
			\mathcal{L}^d(B(x,2\delta)) &\geq& \mathcal{L}^d(\mbox{int}(J)) \sum_{{J_\mathbf{u}}\in A(F)}\lvert J_\mathbf{u}\rvert ^d  \\
			&\geq & c_*^d \#A(F)\delta^d \mathcal{L}^d(\mbox{int}(J)).
		\end{eqnarray*}
		By setting
		$$
		C=\frac{2^d \mathcal{L}^d(B(0,1))}{c_*^d \mathcal{L}^d(\mbox{int}(J))},
		$$
		we have that $\#A(F)\le C$, and the conclusion holds.
	\end{proof}

	\begin{proof}[Proof of Theorem \ref{thm_1}]
		Given $\theta \in (0,1]$, for  the upper intermediate dimension,  it is equivalent to show that  $\uid E\leq s^\theta$ and $\uid E \geq s^\theta$.
		
		First, we prove that $\uid E\leq s^\theta$. Arbitrarily choosing $\beta>\gamma>s^\theta$, for each $\epsilon>0$, there exists $\Delta_1>0$ such that  for all $0<\delta<\Delta_1$, we have
\begin{equation}\label{ineq_dbg<e}
		\frac{\delta^{\beta-\gamma}}{c_*^\beta} <\epsilon.
\end{equation}
Recall that  $s^\theta = \limsup_{\delta \to 0}s_{\delta, \theta}$ where $s_{\delta, \theta}$ is given by \eqref{def_stheta}, and there exists $\Delta_2 >0$ such that   for all  $0<\delta <\Delta_2$, we have that  $\gamma>s_{\delta , \theta} $.
		Moreover,  there exists a cut set $\mathcal{M}_\delta$ such that  $\delta^{\frac{1}{\theta}} < |J_{\mathbf{u}^*}| $  and $ |J_\mathbf{u}|\le \delta$  for every $\mathbf{u} \in \mathcal{M}_\delta $ and satisfying
\begin{equation}\label{eq_cvst}
		\sum_{\mathbf{u}\in \mathcal{M}_\delta} |J_\mathbf{u}|^{s_{\delta , \theta}} =1 .
\end{equation}
It implies  that $\sum_{\mathbf{u}\in \mathcal{M}_\delta} |J_\mathbf{u}|^{\gamma} <1 .$
	
	Let  $\Delta=\min\{\Delta_1 , \Delta_2\}$. For all  $\delta<\Delta$, let $\mathcal{M}_\delta$ be the  cut set given by \eqref{eq_cvst}. By lemma \ref{delta cover},  there exists a cover $\mathcal{F}_\delta=\{U_\bu: \bu\in \mathcal{M}_\delta\}$ of $E$ such that $J_\mathbf{u} \subset U_\mathbf{u}$, $\delta^{\frac{1}{\theta}}\leq |U_\mathbf{u}|\le\delta$ and $|J_\mathbf{u}|\leq |U_\mathbf{u}|<|J_{\mathbf{u}^*}|$ for all $\bu\in \mathcal{M}_\delta$.
Combining with \eqref{ineq_dbg<e} and \eqref{eq_cvst}, we have that 
     \begin{eqnarray*}
			\sum_{U\in \mathcal{F}_\delta} | U |^{\beta}&\leq&\sum_{\mathbf{u}\in \mathcal{M}_\delta} | U_\mathbf{u} |^{\gamma}\delta^{\beta-\gamma} \\
			&\leq &\sum_{\mathbf{u}\in \mathcal{M}_\delta} | J_{\mathbf{u}^*} |^{\gamma}\delta^{\beta-\gamma} \\
			&\leq &\frac{\sum_{\mathbf{u}\in \mathcal{M}_\delta} | J_{\mathbf{u}} |^{\gamma}}{c_*^\gamma}\delta^{\beta-\gamma}  \\
			&\leq &\frac{\delta^{\beta-\gamma}}{c_*^\beta} \\
            & <& \epsilon.
		\end{eqnarray*}
This implies that $\uid E \le \beta$.  Since $\beta \geq s^\theta$ is arbitrarily chosen, we obtain that
		$$
		\uid E \le s^\theta.
		$$
		
		Next, we prove that  $\uid E\geq s^\theta$.  Arbitrarily choosing $\alpha<s^\theta$, recall that
		$$
		s^\theta = \limsup_{\delta \to 0}s_{\delta , \theta},
		$$
		and there exists a sequence $\{\delta_k\}_{k=1}^\infty$ convergent to $0$ such that  $s_{\delta_k, \theta}>\alpha$.

		Fix an integer $k>0$.  For each cover $\mathcal{F}$ of $E$ such that  $ {\delta^\frac{1}{\theta}_k}\le \lvert U\rvert \le \delta_k$ for all $U\in \mathcal{F}$, by Lemma \ref{finite intersection}, there exists  a constant $C$ such that for  every  $U\in \mathcal{F}$ such that  $E\cap U\ne \emptyset$,  we have $\#A(U)\le C$. This implies that
\begin{equation}\label{eq_bKb}
		\sum_{ U\in\mathcal{F}} \sum_{\mathbf{u}\in A(U)} \lvert  J_\mathbf{u}\rvert ^\alpha \le C\sum_{U\in\mathcal{F}} \lvert U\rvert^\alpha.
\end{equation}
By \eqref{def_Aset}, we write that
$$
\Sigma(\mathcal{F}) =\bigcup_ {U\in\mathcal{F}} A(U),
$$
		it is obvious that $\{J_\mathbf{u}:\mathbf{u}\in \Sigma(\mathcal{F})\}$ is a cover of $E$ with  $\delta_k^\frac{1}{\theta} < \lvert  J_{\mathbf{u}^*}\rvert$. Hence we may choose a finite cut set $\{\mathbf{u}_i\}_{i = 1}^{n}\subset\Sigma(\mathcal{F})$,  and  there exists $t\ge s_{\delta_k, \theta}$ such that
\begin{equation}\label{eq_t=1}
\sum_{i = 1}^{n} \lvert J_{\mathbf{u}_i}\rvert^{t} =1.
\end{equation}
Since $t\ge s_{\delta_k, \theta}>\alpha$, it follows that 
$$
\sum_{i = 1}^{n} |J_{\mathbf{u}_i}|^{t} \le  \sum_{i = 1}^{n} \lvert  J_{\mathbf{u}_i}\rvert^\alpha   \le  \sum_{ U\in\mathcal{F}} \sum_{\mathbf{u}\in A(U)} \lvert  J_\mathbf{u}\rvert ^\alpha ,
$$
and combining it with \eqref{eq_bKb} and \eqref{eq_t=1}, we obtain that
$$	
 \sum_{U\in\mathcal{F}} \lvert U\rvert^\alpha\geq \frac{1}{C}.
$$
Setting  $\epsilon_0=\frac{1}{C}$, for every $\Delta>0$,  there exists $\delta_k<\Delta$ such that for every cover $\mathcal{F}$  satisfying that  $ {\delta^\frac{1}{\theta}_k}\le \lvert U\rvert \le \delta_k$ for all $U\in \mathcal{F}$,    we have that
$$
		\sum_{U\in\mathcal{F}} \lvert U\rvert^\alpha \geq \epsilon_0.
$$
It follows that
		$$
		\uid E \geq \alpha.
		$$
		Since $\alpha<s^\theta $ is arbitrarily chosen, 		we have that $ \uid E \geq s^\theta$.
		
The proof 	for the lower intermediate dimension is almost identical to  the upper intermediate dimension, and we leave it to the readers as an exercise.	
	\end{proof}

Given a Moran set $E$. Let   $\mathcal{M}$ be a cut set   with  $\#\mathcal{M}<\infty$. We write
	$$
	L_\mathcal{M} =\min\{|\mathbf{u}|: \mathbf{u} \in \mathcal{M}\}, \qquad \qquad
	K_\mathcal{M} =\max\{|\mathbf{u}|: \mathbf{u} \in \mathcal{M}\}.
	$$

	\begin{lem}\label{lem_1}
		Let $E$ be a Moran set given by \eqref{attractor}. Then  for every cut set  $\mathcal{M}$  with  $\#\mathcal{M}<\infty$,  we have
		$$
		\sum_{\mathbf{u}\in \mathcal{M}} \lvert J_\mathbf{u}\rvert ^\beta  > 1,
		$$
		for all   $\beta<\min_{L_\mathcal{M} \le k\le K_\mathcal{M} } s_k$, where $s_k$ is given by \eqref{s1}.
		
	\end{lem}
	
	\begin{proof}
		Since  $\mathcal{M}$ is a cut set   with  $\#\mathcal{M}<\infty$, we have that
		\begin{equation}\label{MJB}
			\sum_{\mathbf{u}\in \mathcal{M}} \lvert J_\mathbf{u}\rvert ^\beta =\sum_{k=L_\mathcal{M}}^{ K_\mathcal{M}}\sum_{\mathbf{u} \in \mathcal{M}, \mathbf{u} \in \Sigma^k}\lvert J_\mathbf{u}\rvert ^\beta.
		\end{equation}
		Note that for each $\mathbf{u}\in \mathcal{M}$ such that $|\mathbf{u}|= K_\mathcal{M}$, it is clear that $\mathbf{u}^* j \in \mathcal{M}$  for each $1\le j\le n_{K_\mathcal{M}}$. Since  $\beta<\min_{L_\mathcal{M} \le k\le K_\mathcal{M} } s_k$, it immediately follows that
		$$
		\lvert J_{\mathbf{u}^*}\rvert^\beta \sum_{j=1}^ {n_{K_\mathcal{M}}} c_{K_\mathcal{M},j}^{s_{K_\mathcal{M}}} < \lvert J_{\mathbf{u}^*}\rvert^\beta \sum_{j=1}^ {n_{K_\mathcal{M}}} c_{K_\mathcal{M},j}^\beta =\sum_{j=1}^ {n_{K_\mathcal{M}}} \lvert J_{\mathbf{u}^* j}\rvert^\beta .
		$$
		Let $\Lambda =\{ \mathbf{u}\in \mathcal{M} : |\mathbf{u}|= K_\mathcal{M} \}$ and $\Lambda^* =\{ \mathbf{u}^* : \mathbf{u}^*j\in\Lambda  \textit{ for some  } j=1,2,\ldots K_\mathcal{M} \}$ . Since $\beta<s_{K_\mathcal{M}}$, we have
		$$
		\Big(\sum_{\mathbf{u}^*\in \Lambda^*} \lvert J_{\mathbf{u}^*}\rvert^\beta\Big)\Big(\sum_{j=1}^ {n_{K_\mathcal{M}}} c_{K_\mathcal{M},j}^{s_{K_\mathcal{M}}}\Big)< \Big(\sum_{\mathbf{u}^*\in \Lambda^*} \lvert J_{\mathbf{u}^*}\rvert^\beta\Big)\Big(\sum_{j=1}^ {n_{K_\mathcal{M}}} c_{K_\mathcal{M},j}^\beta\Big) =			\sum_{\mathbf{u} \in \Lambda}\lvert J_\mathbf{u}\rvert ^\beta.
		$$

		To show $\sum_{\mathbf{u}\in \mathcal{M}} \lvert J_\mathbf{u}\rvert ^\beta >  1$, we need go through the following process inductively.		Let
		$$
		\Lambda_1 =\{\mathbf{u}: \textit{either } \mathbf{u} j \in \Lambda  \textit{ for some $j=1,\ldots ,  n_{K_\mathcal{M}}$ or } \mathbf{u}\in \mathcal{M} \textit{ such that }   |\mathbf{u}|=K_\mathcal{M}-1\}.
		$$
		If
		$$
		\sum_{j=1}^{ n_{K_\mathcal{M}}} c_{K_\mathcal{M},j}^{s_{K_\mathcal{M}}}\geq 1,
		$$
		then it is clear that
		\begin{eqnarray*}
			\sum_{k=K_\mathcal{M}-1}^{ K_\mathcal{M}}\sum_{\mathbf{v} \in \mathcal{M}\cap \Sigma^k}\lvert J_\mathbf{v}\rvert ^\beta&=&  	\sum_{\mathbf{u}\in\Lambda^*} \sum_{j=1}^{n_{K_\mathcal{M}}}\lvert J_{\mathbf{u}j}\rvert ^\beta   +\sum_{\mathbf{u}\in\Lambda_1 \setminus\Lambda^*} |J_\mathbf{u}|^\beta   \\
			&=& 	\Big(\sum_{\mathbf{u}\in \Lambda^*} \lvert J_{\mathbf{u}}\rvert^\beta\Big)\Big(\sum_{j=1}^{ n_{K_\mathcal{M}}} c_{K_\mathcal{M},j}^\beta \Big) +\sum_{\mathbf{u}\in\Lambda_1 \setminus\Lambda^*} |J_\mathbf{u}|^\beta    \\
			&>& 	\Big(\sum_{\mathbf{u}\in \Lambda^*} \lvert J_{\mathbf{u}}\rvert^\beta\Big)\Big(\sum_{j=1}^{ n_{K_\mathcal{M}}} c_{K_\mathcal{M},j}^{s_{K_{\mathcal{M}}}} \Big) +\sum_{\mathbf{u}\in\Lambda_1 \setminus\Lambda^*} |J_\mathbf{u}|^\beta     \\
			&\geq& 	\sum_{\mathbf{u}\in\Lambda_1} \lvert J_{\mathbf{u}}\rvert^\beta.
		\end{eqnarray*}
		We write $\mathcal{M}'= \{\mathbf{u}:|\mathbf{u}|<K_\mathcal{M} \textit{ and } \mathbf{u}\in \mathcal{M}\cup\Lambda_1\}$, and by \eqref{MJB}, it follows that
		$$
		\sum_{\mathbf{u}\in \mathcal{M}} \lvert J_\mathbf{u}\rvert ^\beta>		\sum_{\mathbf{u}\in \mathcal{M}'} \lvert J_\mathbf{u}\rvert ^\beta,
		$$
		We replace  $\mathcal{M}$ by $\mathcal{M}'$ and  repeat above process.
		Otherwise if
		$$
		\sum_{j=1}^ {n_{K_\mathcal{M}}} c_{K_\mathcal{M},j}^{s_{K_\mathcal{M}}}< 1,
		$$
		then we add descendants for every elements $\mathbf{u}\in \mathcal{M} \cap \Sigma^{K_\mathcal{M} -1} $ and have that
		$$
		\sum_{k=K_\mathcal{M}-1}^{ K_\mathcal{M}}\sum_{\mathbf{u} \in \mathcal{M}\cap \Sigma^k}\lvert J_\mathbf{u}\rvert ^\beta \geq \sum_{\mathbf{u}\in\Lambda_1} \lvert J_{\mathbf{u}}\rvert^\beta   \sum_{j=1}^ {n_{K_\mathcal{M}}} c_{K_\mathcal{M},j}^{s_{K_\mathcal{M}}}.
		$$
		For $\mathbf{u}\in\Lambda_1$ , we have
		$$
		\lvert J_{\mathbf{u}^*}\rvert^\beta \sum_{j=1}^ {n_{K_\mathcal{M}-1}} c_{K_\mathcal{M}-1,j}^{s_{K_\mathcal{M}-1}} < \lvert J_{\mathbf{u}^*}\rvert^\beta \sum_{j=1}^ {n_{K_\mathcal{M}-1}} c_{K_\mathcal{M}-1, j}^\beta =\sum_{j=1}^ {n_{K_\mathcal{M}-1}} \lvert J_{\mathbf{u}^* j}\rvert^\beta .
		$$
		Let $\Lambda_1^* =\{ \mathbf{u}^* : \mathbf{u}^*j\in\Lambda _1 \textit{ for some  } j=1,2,\ldots , K_\mathcal{M}-1 \}$.  Since $\beta<s_{K_\mathcal{M}-1}$, it follows that
		\begin{eqnarray*}
			\Big(\sum_{\mathbf{u}^*\in\Lambda_1^*} \lvert J_{\mathbf{u}^*}\rvert^\beta\Big)\Big(\sum_{j=1}^ {n_{K_\mathcal{M}}} c_{K_\mathcal{M},j}^{s_{K_\mathcal{M}}}\Big)\Big(\sum_{j=1}^{ n_{K_\mathcal{M} -1}} c_{K_\mathcal{M}-1,j}^{s_{K_\mathcal{M}-1}}\Big) &<& \Big(\sum_{\mathbf{u} \in \Lambda_1} \lvert J_\mathbf{u}\rvert ^\beta\Big)\Big(\sum_{j=1}^ {n_{K_\mathcal{M}}} c_{K_\mathcal{M},j}^{s_{K_\mathcal{M}}}\Big)  \\
			&<&\sum_{K_\mathcal{M}-1\le k\le K_\mathcal{M}}\sum_{\mathbf{u} \in \mathcal{M}\cap \Sigma^k}\lvert J_\mathbf{u}\rvert ^\beta.
		\end{eqnarray*}
		Let
		$$
		\Lambda_2 =\{\mathbf{u}: \textit{either } \mathbf{u} j \in \Lambda_1  \textit{ for some $j=1,\ldots  , n_{K_\mathcal{M}-1}$ or } \mathbf{u}\in \mathcal{M} \textit{ such that }   |\mathbf{u}|=K_\mathcal{M}-2\}.
		$$
		If
		$$
		1\le \sum_{j=1}^ {n_{K_\mathcal{M}}} c_{K_\mathcal{M},j}^{s_{K_\mathcal{M}}}\sum_{j=1}^{ n_{K_\mathcal{M} -1}} c_{K_\mathcal{M}-1,j}^{s_{K_\mathcal{M}}},
		$$
 we similarly have that
		$$
		\sum_{k=K_\mathcal{M}-2}^{ K_\mathcal{M}}\sum_{\mathbf{u} \in \mathcal{M} \cap \Sigma^k}\lvert J_\mathbf{u}\rvert ^\beta  > \sum_{\mathbf{u}\in\Lambda_2} \lvert J_{\mathbf{u}}\rvert^\beta.
		$$
		Otherwise we  continue the same process as the previous discussion. Since $K_\mathcal{M}-L_\mathcal{M}<\infty,$ we  go through the processes at most $K$ times,  where $L_\mathcal{M}\le K \le L_\mathcal{M}$. This implies that
		$$
		\sum_{\mathbf{u}\in \mathcal{M}} \lvert J_\mathbf{u}\rvert ^\beta> \prod_{k=1}^{K} \sum_{j=1}^{n_k} c_{k,j}^{s_{K}}=1,
		$$
and the conclusion follows.
	\end{proof}
	
Recall that for homogenous Moran sets,  the $s_k$ given by \eqref{s1} is simplified  into
\begin{equation} \label{def_hmssk}
	s_k =-\frac{\log n_1 \dots n_k}{\log c_1 \dots c_k}.
\end{equation}
To find the intermediate dimensions of Homogenous Moran sets, we need to show the distance of $s_k$ and $s_{k+1}$ is sufficiently small. 	
	\begin{lem}\label{continue}
		Let $E$ be a homogeneous Moran set.  If $c_*>0$, then
		$$
		\lim_{k\to\infty} (s_k -s_{k+1}) =0,
		$$
		where $s_k$ is given by \eqref{s1}.
	\end{lem}
	\begin{proof}
		Since for every $\mathbf{u}\in \Sigma^k$,
		$$
		\bigcup_{i=1}^{n_{k}}J_{\mathbf{u}i}\subset J_\mathbf{u},
		$$
		where $\mbox{int} (J_{\mathbf{u}i})\cap \mbox{int}(J_{\mathbf{u}i^{\prime }})=\varnothing$,   it follows that
		\begin{eqnarray*}
			\mathcal{L}^d\Big(\mbox{int}\Big(\bigcup_{i=1}^{n_{k}}J_{\mathbf{u}u_i}\Big)\Big) &=& \mathcal{L}^d(\mbox{int}(J_\mathbf{u})) n_{k}{c_k}^d  \\
			&\le& \mathcal{L}^d(\mbox{int} (J_\mathbf{u})).
		\end{eqnarray*}
		This implies that  $s_k \le d$ and { $n_k {c_k}^d \le1$}, and we have that
		\begin{eqnarray*}
			|s_k -s_{k+1} |&=&\Big|-\frac{\log n_1 \dots n_k}{\log c_1 \dots c_k}+\frac{\log n_1 \dots n_{k+1}}{\log c_1 \dots c_{k+1}}  \Big|\\
			&=& \Big| \frac{\log n_{k+1}}{\log c_1 \dots c_{k+1}}-\frac{\log n_1 \dots n_k \log c_{k+1}}{\log c_1 \dots c_{k+1} \log c_1 \dots c_k}  \Big| \\
			&\le&  \Big|\frac{\log n_{k+1}}{\log c_1 \dots c_{k+1}}\Big|+d\Big|\frac{\log c_{k+1}}{\log c_1 \dots c_{k+1}} \Big|.
		\end{eqnarray*}
		Let $k$ tend to $\infty$, and  we have that
		$$
		\lim_{k\to\infty} (s_k -s_{k+1}) =0.
		$$
	\end{proof}

	For a homogeneous Moran set $E$, recall that $l(k,\theta)$ is given by
$$  
		c_1c_2\ldots c_l\leq (c_1c_2\ldots c_k)^{\frac{1}{\theta}}<c_1c_2\ldots c_{l-1}.
$$ 	
\begin{prop} \label{Prop_HMS}
		Let $E$ be a homogeneous Moran set with $c_*>0$. Then we have
		$$
		s^{\theta}=\limsup_{k\to\infty} \min_{k\leq m\leq l(k,\theta)} s_m,
		$$
where $s_m$ is given by \eqref{def_hmssk}.
	\end{prop}
	
	\begin{proof}
		Without loss of generality, we assume that $|J|=1$. For every $\delta<c_1$,  there exist intergers $k(\delta)$ and $l(\delta)$ such that
		$$
		c_1c_2\ldots c_{k(\delta)}=|J_\mathbf{u}|\le \delta <|J_{\mathbf{u}^*}|=c_1c_2\ldots c_{k(\delta)-1},
		$$
for all  $\mathbf{u}\in\Sigma^{k(\delta)}$	and
		$$
		c_1c_2\ldots c_{l(\delta)}=|J_\mathbf{v}|\le\delta^{\frac{1}{\theta}} < |J_{\mathbf{v}^*}|=c_1c_2\ldots c_{l(\delta)-1},
		$$
for all $\mathbf{v}\in\Sigma^{l(\delta)}.$	

By the definitions of $s_k$ and $s_{\delta, \theta}$,  we have
		$$
		\min_{k(\delta)\le m\le l(\delta)} s_m\geq s_{\delta, \theta}.
		$$
		If $\min_{k(\delta)\le m\le l(\delta)} s_m> s_{\delta, \theta}$, then there exists a cut-set $\mathcal{M}$ satisfying
		$$
		\sum_{\mathbf{u}\in\mathcal{M}} |J_\mathbf{u}|^{s_{\delta, \theta}}=1 ,
		$$
where $ k(\delta)\le|\mathbf{u}|\le l(\delta)$ for all $\mathbf{u}\in\mathcal{M}$.  This contradicts  Lemma \ref{lem_1}, and we have
		$$
		\min_{k(\delta)\le m\le l(\delta)} s_m = s_{\delta, \theta}.
		$$

For each integer $k>0$, recall that   $l(k,\theta)$ is given by
$$
c_1c_2\ldots c_{l(k,\theta)}\leq (c_1c_2\ldots c_k)^{\frac{1}{\theta}}<c_1c_2\ldots c_{l(k,\theta)-1}.
$$ 
Let $k=k(\delta)$. Since
$$
(c_1c_2\ldots c_k)^{\frac{1}{\theta}} \le\delta^\frac{1}{\theta}  \le (c_1c_2\ldots c_{k-1})^{\frac{1}{\theta}},
$$
it implies that
$$
l(k,\theta)\ge l(\delta)>l(k-1, \theta).
$$
 Hence we have
\begin{equation}\label{sd=sm}
\min_{k\leq m\leq l(k,\theta)} s_m\le \min_{k\le m\le l(\delta)} s_m = s_{\delta, \theta}<\min_{k\leq m\leq l(k-1,\theta)} s_m.
\end{equation}

Since
$$
0\le \min_{k+1\leq m\leq l(k,\theta)} s_m -\min_{k\leq m\leq l(k,\theta)} s_m  \le|s_{k+1}-s_k|,
$$
by Lemma \ref{continue},  we have
$$
\lim_{k\to\infty} \big(\min_{k+1\leq m\leq l(k,\theta)} s_m -\min_{k\leq m\leq l(k,\theta)} s_m \big)=0.
$$
It is follows that
		$$
		\limsup_{k\to\infty} \min_{k+1\leq m\leq l(k,\theta)} s_m=\limsup_{k\to\infty} \min_{k\leq m\leq l(k,\theta)}s_m.
		$$
Therefore, by \eqref{sd=sm}, we have
		$$
		s^\theta=\limsup_{\delta\to0} s_{\delta , \theta}=\limsup_{k\to\infty} \min_{k\leq m\leq l(k,\theta)} s_m,
		$$
and the conclusion holds.
	\end{proof}
	
	\begin{proof}[Proof of Corollary \ref{cor_HMS}]
		The conclusion follows directly  from Theorem \ref{thm_1} and Proposition \ref{Prop_HMS}.
	\end{proof}
	
	\section{Intermediate dimension of Moran sets with $c_*=0$}
	In the section, we study the intermediate dimensions of Moran sets $E$ with $c_*=0$. Given a set $F\subset \R^d$ such that $E\cap F\ne \emptyset$,  recall that
	$$
	A(F)=\{\mathbf{u}:\mathbf{u}\in \mathcal{M}(|F|) , J_\mathbf{u}\cap F \ne \emptyset\},
	$$
	where  $\mathcal{M}(|F|) =\{\mathbf{u}\in\Sigma^*: |J_\mathbf{u}|\le |F|<|J_{\mathbf{u}^*}| \}$.	Since $c_*=0$, the number of elements in $A(F)$ is not necessarily bounded with respect to $F$, which is important in the dimension estimation. To overcome this obstacle, we have to further classify the set $A(F)$. Let
\begin{equation}\label{def_k0}
k_0 = \min\{k:|\mathbf{u}|=k , \mathbf{u}\in A(F)\},
\end{equation}
and for each integer $k\geq k_0$, we write
	$$D(F, k)=\{\mathbf{u} \in \Sigma^k :\mathbf{u} \in A(F)\}.
	$$

	In the following conclusion, we show that the number of element in $D(F, k)$ does not increase very fast under certain restrictions on $\underline{c}_k$ and $M_k$ where
	$$
	\underline{c}_k=\min_{1\le j \le n_k} \{c_{k,j}\}, \quad  \textit{and }\quad M_k=\max_{\mathbf{u} \in \Sigma^k} \lvert J_\mathbf{u} \rvert,
	$$
\begin{lem}\label{finite intersection c_*=0}
		Given a Moran set  $E$  with $c_*=0$. Suppose that
		$$
		\lim_{k\to +\infty} \frac{\log{\underline{c}_k}}{\log{M_k}}=0.
		$$
		Then there exists  a constant $C$ such that for  every  $F \subset \mathbb{R}^d$ with  $E\cap F\ne \emptyset$, we have
		$$
		\sum_{k = k_0}^{\infty}\underline{c}_k^d \# D(F, k)\le C,
		$$
where $k_0$ is given by \eqref{def_k0}.
	\end{lem}

	\begin{proof}
		Given a set $F \subset \mathbb{R}^d$ such that   $E\cap F\ne \emptyset$. For every $\bu\in \mathcal{M}(|F|) $, it is clear that
		$$
		\underline{c}_{|\mathbf{u}|}\lvert F\rvert \le \underline{c}_{|\mathbf{u}|}\lvert J_{\mathbf{u}^*} \rvert \le \lvert J_{\mathbf{u}} \rvert,
		$$
		Arbitrarily choose $x\in F$,  and  we have that   $J_\mathbf{u} \subset B(x,2\lvert F\rvert)$ for every $ J_\mathbf{u} \in A(F)$. It immediately follows that
		\begin{eqnarray*}
			\sum_{k = k_0}^{\infty}\underline{c}_k^d\# D(F, k)\lvert F\rvert ^d \mathcal{L}^d((\mbox{int}J))&\le& \mathcal{L}^d(\mbox{int}(J))\sum_{k = k_0}^{\infty}\sum_{\mathbf{u} \in D(F, k)} \lvert J_\mathbf{u}\rvert ^d   \\
			&=&\mathcal{L}^d(\mbox{int}(J)) \sum_{\mathbf{u} \in A(F)}\lvert J_\mathbf{u}\rvert ^d  \\
			&\le& \mathcal{L}^d(B(x,2\lvert F\rvert).
		\end{eqnarray*}
		Hence we obtain that
		$$
		\sum_{k = k_0}^{\infty}\underline{c}_k^d \# D(F, k)\le \frac{2^d \mathcal{L}^d(B(0,1))}{ \mathcal{L}^d((\mbox{int}J))},
		$$
		and the conclusion holds by setting   $C=\frac{2^d \mathcal{L}^d(B(0,1))}{ \mathcal{L}^d((\mbox{int}J))}$.
	\end{proof}
	\begin{proof}[Proof of Theorem \ref{thm_3}]
We only given the proof for the lower intermediate dimension since the proof for upper intermediate dimension is similar. 	For the lower intermediate dimension, it is equivalent to show that  $\lid E\le s_\theta$ and $\lid E\ge s_\theta$.
		
		First, we prove $\lid E\geq s_\theta$. Arbitrarily choose $\alpha<s_\theta$. Since  $s_\theta =\liminf_{\delta \rightarrow 0 } s_{\delta , \theta}$, there exists $\Delta_1 >0$  such that   for all  $0<\delta <\Delta_1$, we have that  $\alpha<s_{\delta , \theta} $.  Since $\lim_{k\to +\infty} \frac{\log{\underline{c}_k}}{\log{M_k}}=0$, for each $\eta>0$, there exists $K_0>0$, such that  when $k>K_0$, we have
\begin{equation}\label{Mk<ck1}
		M_k^\eta<\underline{c}_k^d,
\end{equation}
and there exists $\Delta_2$ such that for all $0<\delta<\Delta_2$,  we have $|\mathbf{u}|>K_0$ for all $\bu\in\Sigma^*$ satisfying $|J_\mathbf{u}|\le\delta$. 

Given a  cover  $\mathcal{F}$ of $E$ such that  $ {\delta^\frac{1}{\theta}}\le | U| \le \delta$ for each $U\in \mathcal{F}$.  By Lemma \ref{finite intersection c_*=0},  there exists  a constant $C$ such that for  every  $U\in \mathcal{F} $ with  $E\cap U\ne \emptyset$, we have
\begin{equation}\label{sum<C1}
		\sum_{k = k_0}^{\infty}\underline{c}_k^d \# D(U, k)\le C,
\end{equation}
where $k_0$ is given by \eqref{def_k0}. Since $|J_\bu|\leq |U|$ for every $\mathbf{u} \in D(U, k)$, combining \eqref{Mk<ck1} and \eqref{sum<C1} together,  we have for $k_0 >K_0$
		\begin{eqnarray*}
	\sum_{ U\in\mathcal{F}} \sum_{\mathbf{u}\in A(U)} \lvert  J_\mathbf{u}\rvert ^\alpha  &=&\sum_{U\in\mathcal{F}} \sum_{k = k_0}^{\infty}\sum_{\mathbf{u} \in D(U, k)} |J_\mathbf{u}|^\alpha   \\
			&\le& \sum_{U\in\mathcal{F}} \sum_{k = k_0}^{\infty}\sum_{\mathbf{u} \in D(U, k)}  M_k ^\eta \lvert U\rvert^{\alpha-\eta} \\
			&\le& \sum_{U\in\mathcal{F}} |U|^{\alpha-\eta} \sum_{k = k_0}^{\infty}  \underline{c}_k^d  \# D(U, k)   \\
			&\le& C\sum_{U\in\mathcal{F}} \lvert U\rvert^{\alpha-\eta}.
		\end{eqnarray*}
		
Let $	\mathcal{F_1} =\{J_\mathbf{u}:\mathbf{u}\in A(U), U\in\mathcal{F}\}$ and $\mathcal{F_1}$ is a cover of $E$ satisfying $\delta^\frac{1}{\theta} < | J_{\mathbf{u}^*}|$ and $| J_{\mathbf{u}}|\leq \delta$. Moreover, we may choose a finite cut set
$$
\{{\mathbf{u}_i}\}_{i = 1}^{n}\subset \bigcup_{U\in\mathcal{F}} A(U)
$$
such that  $\{J_{\mathbf{u}_i}\}_{i = 1}^{n}\subset\mathcal{F_1}$ is a cover of $E$, and  there exists $t\ge s_{\delta,\theta}>\alpha$  such that
$$
\sum_{i = 1}^{n} \lvert J_{\mathbf{u}_i}\rvert^{t} =1.
$$
Since  $t\ge s_{\delta,\theta}>\alpha$,  we have that
$$
\sum_{ U\in\mathcal{F}} \sum_{\mathbf{u}\in A(U)} \lvert  J_\mathbf{u}\rvert ^\alpha  \geq \sum_{i = 1}^{n} | J_{\mathbf{u}_i}|^\alpha \geq \sum_{i = 1}^{n} \lvert J_{\mathbf{u}_i}\rvert^{t}=1,
$$		
and it implies that
$$
\sum_{U\in\mathcal{F}} \lvert U\rvert^{\alpha-\eta}\geq \frac{1}{C}.
$$

Setting $\epsilon_0=\frac{1}{C}$ and $\Delta_0=\min\{\Delta_1, \Delta_2\}$,  for every $0<\delta<\Delta_0$  and every cover $\mathcal{F}$ satisfying  $ {\delta^\frac{1}{\theta}}\le | U| \le \delta$ for all $U\in \mathcal{F}$,    we have that
		$$
		\sum_{U\in\mathcal{F}} \lvert U\rvert^{\alpha-\eta} \geq \epsilon_0.
		$$
		It implies that $ \lid E \ge \alpha-\eta$. Since $\eta >0$ and $\alpha < s^\theta$ are  arbitrarily chosen, we obtain that
		$$
		 \lid E \ge s_\theta,
		$$

		Next, we prove $\lid E\leq s_\theta$. Arbitrarily choose $\beta>\gamma>s_\theta$.
Since $\gamma>s_\theta$,  there exists a sequence $\{\delta_k\}_{k=1}^\infty$ convergent to $0$ such that  $\gamma>s_{\delta_k,\theta}$. Moreover,  for each $k>0$, there exists a  cut set $\mathcal{M}_k$ such that $\delta_k^{\frac{1}{\theta}} < |J_{\mathbf{u}^*}|$   and $ |J_\mathbf{u}|\le \delta_k$ for all $\mathbf{u} \in \mathcal{M}_k$  satisfying
\begin{equation}\label{skd=1}
\sum_{\mathbf{u}\in \mathcal{M}_k} |J_\mathbf{u}|^{s_{\delta_k,\theta}} =1 .
\end{equation}
Moreover, by lemma \ref{delta cover},  there exists a cover $\mathcal{F}_k=\{U_\bu: \bu\in \mathcal{M}_k\}$ of $E$ such that $J_\mathbf{u} \subset U_\mathbf{u}$, $\delta_k^{\frac{1}{\theta}}\leq |U_\mathbf{u}|\le\delta_k$ and $|J_\mathbf{u}|\leq |U_\mathbf{u}|<|J_{\mathbf{u}^*}|$ for all $\bu\in \mathcal{M}_k$.

Since $\{\delta_k\}_{k=1}^\infty$ is convergent to $0$, for each $ \epsilon >0$ and  $\Delta>0$, there exists  an integer  $K_1>0$ such that for all $k>K_1$, we have that $\delta_k <\Delta$ and
\begin{equation}\label{de<ep}
\delta_k^{\frac{\beta-\gamma}{2}} <\epsilon.
\end{equation}

Since
$$
\lim_{k\to +\infty} \frac{\log{\underline{c}_k}}{\log{M_k}}=0,
$$
there exists an integer  $K_2>0$, such that  for all  $k>K_2$,
\begin{equation}\label{MC<1}
		\frac{M_k^{\frac{\beta-\gamma}{2}}}{\underline{c}_k^{s_{\delta_k,\theta}+\beta-\gamma}}<\frac{M_k^{\frac{\beta-\gamma}{2}}}{\underline{c}_k^{\beta}}<1.
\end{equation}

Hence for all $\epsilon>0$ and $\Delta>0$, choose $k>\max\{K_1,K_2\}$,  and $\mathcal{F}_k=\{U_\bu: \bu\in \mathcal{M}_k\}$ is a cover  of $E$ satisfying that $J_\mathbf{u} \subset U_\mathbf{u}$, $\delta_k^{\frac{1}{\theta}}\leq |U_\mathbf{u}|\le\delta_k$ and $|J_\mathbf{u}|\leq |U_\mathbf{u}|<|J_{\mathbf{u}^*}|$ for all $\bu\in \mathcal{M}_k$.  Since $\beta>\gamma>s_\theta$, by \eqref{de<ep}, \eqref{MC<1} and \eqref{skd=1},  we obtain that
		\begin{eqnarray*}
			\sum_{\mathbf{u}\in \mathcal{M}_k} | U_\mathbf{u} |^{\beta}
			&\leq &\sum_{\mathbf{u}\in \mathcal{M}_k} | J_{\mathbf{u}^*} |^{s_{\delta_k,\theta}+\beta-\gamma} \\
			&\leq &\sum_{\mathbf{u}\in \mathcal{M}_k} \Big(\frac{| J_{\mathbf{u}} |}{\underline{c}_{|\mathbf{u}|}}\Big)^{s_{\delta_k,\theta}+\beta-\gamma} \\
			&\leq &\sum_{\mathbf{u}\in \mathcal{M}_k} | J_{\mathbf{u}} |^{s_{\delta_k,\theta}+\frac{\beta-\gamma}{2}}    \Big(\frac{M_{|\mathbf{u}|}^{\frac{\beta-\gamma}{2}}}{\underline{c}_{|\mathbf{u}|}^{s_{\delta_k,\theta}+\beta-\gamma}}\Big) \\
			&\leq&\delta^\frac{\beta-\gamma}{2} \\
&<&\epsilon
		\end{eqnarray*}
		It follows that $\lid E \le \beta$.  Since $\beta \geq s^\theta$ is arbitrarily chosen, we obtain that
		$$
		\lid E \le s^\theta.
		$$
\end{proof}

Next, we study the intermediate dimension of homogeneous Moran sets with $c_*=0$. The key idea is similar to the case with $c_*>0$, and the following conclusions are  the same as before with extra assumptions. Since the proofs are similar, we  only give the key argument in the  proofs to show the difference. The next result shows the distance between $s_k$ and $s_{k+1}$ tends to $0$. 	
	\begin{lem}\label{continue 0}
		Let $E$ be a homogeneous Moran set with $c_*=0$. Suppose that
		$$
		\lim_{k\to +\infty} \frac{\log{c_k}}{\log{c_1 \dots c_k}}=0.
		$$
		Then
		$$
		\lim_{k\to\infty} s_k -s_{k+1} =0,
		$$
		where $s_k$ is given by \eqref{def_hmssk}.
	\end{lem}
	
	\begin{proof}
		Since  $s_k \le d$ and $n_k {c_k}^d <1$, we have that
		\begin{eqnarray*}
			\Big|s_k -s_{k+1} \Big|&\le&  \Big|\frac{\log  n_{k+1}}{\log c_1 \dots c_{k+1}}+d\frac{\log c_{k+1}}{\log c_1 \dots c_{k+1}} \Big|.
		\end{eqnarray*}
		The fact $\lim_{k\to +\infty} \frac{\log{c_k}}{\log{c_1 \dots c_k}}=0$ implies that
		$$
		\lim_{k\to\infty} s_k -s_{k+1} =0.
		$$
	\end{proof}

	\begin{prop} \label{Prop_HMS 0}
		Let $E$ be a homogeneous Moran set with $c_*=0$. If
		$$
		\lim_{k\to +\infty} \frac{\log{c_k}}{\log{c_1 \dots c_k}}=0,
		$$
		then we have
		$$
		s^{\theta}=\limsup_{k\to\infty} \min_{k\leq m\leq l(k,\theta)} s_m,
		$$
where $s_k$ is given by \eqref{def_hmssk}.
	\end{prop}
	The proof is  the same as proposition \ref{Prop_HMS}, and we omit it.
	
	\begin{proof}[Proof of Corollary \ref{thm_4}]
		The  conclusion follows directly from  Theorem \ref{thm_3} and Proposition \ref{Prop_HMS 0}.
	\end{proof}

	\section{Visualization and examples of Intermediate dimensions} \label{sec_VEID}
In this section, we first show the visualization of a class of homogeneous Moran sets. Then we give some examples to illustrate our main conclusions.

Given an integer $L\ge 2$, recall that
	$$
	\mathcal{F}_L=\Big \{f(\theta)=\frac{La \theta +b}{Lc \theta +c} : a,b,c\in\mathbb{R}, c>a\ge b>0, \frac{a}{b}\in\mathbb{N} \Big\}.
	$$
We show that for every $f\in \mathcal{F}_L$,  there exists a Moran set such that $\uid E=f(\theta)$.
	\begin{proof}[Proof of Proposition \ref{Prop_MF}]
Given
$$
f(\theta)=\frac{La \theta +b}{Lc \theta +c} \in  \mathcal{F}_L,
$$
where $c>a\ge b>0$ such that $\frac{a}{b}\in\mathbb{N}$.

Let  $\alpha ,\beta,\gamma\in\mathbb{R}$ satisfy that $a=\log\alpha$,  $b=\log\beta$, $c=\log\gamma$. Since $\beta>1$, there exists a real $l>0$ such that $\beta^l>2$ is an integer.  Since
		$$
		\frac{\log\alpha^l}{\log\beta^l}= \frac{a}{b}\in\mathbb{N},
		$$
it is clear that  $\alpha^l\in\mathbb{N}$ and $\alpha^l>\beta^l$.
By setting $M=\alpha^l, N=\beta^l$ and $Q=\gamma^l$,  we have that  $M>N$ and
		$$
		f(\theta)=\frac{L \theta \log M  +\log N}{L\theta \log Q  +\log Q}.
		$$

Let $n_1=N$ and $c_k =\frac{1}{Q}$ for all $k>0$.   For every  $k\geq 2$,  we write that
		\begin{align*}
			n_k=
			\begin{cases}
				N, \quad L+L^2+\dots+L^{2n-2}<k\le L+L^2+\dots+L^{2n-1} ;\\
				M,\quad L+L^2+\dots+L^{2n-1}<k\le L+L^2+\dots+L^{2n}.
			\end{cases}
		\end{align*}

		Let $E$ be the  corresponding homogeneous Moran set  given by \eqref{attractor}.  By Corollary \ref{cor_HMS}, we have that
$$
\uid E=\limsup_{k\to\infty} \min_{k\leq m\leq l(k,\theta)} s_m,
$$
where
\begin{equation}\label{s_mdef}
 s_m= \frac{\log n_1 \dots n_m}{m\log Q}
\end{equation}
 and $l(k,\theta)$ is given by
$$
Q^{-l(k,\theta)}\leq Q^{-\frac{k}{\theta}}<Q^{-(l(k,\theta)-1)}.
$$

It suffices to prove that
$$
f(\theta)=\limsup_{k\to\infty} \min_{k\leq m\leq l(k,\theta)} s_m,
$$
for $\theta\in [\frac{1}{L^2},1]$.

For each integer $n>0$, let $\mathbf{u}_n, \mathbf{v}_n \in\Sigma^*$ with
\begin{equation} \label{def_uvn}
|\mathbf{u}_n|=\frac{L(1-L^{2n-1})}{1-L}, \qquad |\mathbf{v}_n|=\frac{L(1-L^{2n})}{1-L}.
\end{equation}
Since $L>2$, it is clear that
$$
|\mathbf{u}_1|<|\mathbf{v}_1|<|\mathbf{u}_2|<\ldots<|\mathbf{v}_{n-1}|<|\mathbf{u}_n|<|\mathbf{v}_n|<|\mathbf{u}_{n+1}|<\ldots.
$$
Since  $M\ge N$,  we have that  $s_k$ is monotonically increasing if  $|\mathbf{u}_n|<k\le |\mathbf{v}_n|$ and monotonically decreasing if $|\mathbf{v}_{n-1}|<k\le |\mathbf{u}_n|$.

Given $\theta\in (\frac{1}{L^2}, 1)$, we claim that, for every sufficiently large integer $n>0$, there exists  $\mathbf{z}_n\in\Sigma^*$ satisfying $|\mathbf{u}_n|\le|\mathbf{z}_n|<|\mathbf{u}_{n+1}|$ and
\begin{eqnarray*}
\max_{|\mathbf{u}_n|\le k\le |\mathbf{u}_{n+1}|}\min_{k\leq m\leq l(k,\theta)} s_m &=&\max_{|\mathbf{z}_{n}|-1\le k\le |\mathbf{z}_{n}|+1}\min_{k\leq m\leq l(k,\theta)} s_m.
\end{eqnarray*}

		To prove the claim,   for each $n>0$, we define $f_n:[0,+\infty)\rightarrow \mathbb{R}$ and $g_n:[0,+\infty)\rightarrow \mathbb{R}$
by
		$$
		f_n(x)=\frac{(L^2+L^4+\dots+L^{2n-2}+x)\log M +(L+L^3+\dots+L^{2n-1})\log N}{(L+L^2+\dots+L^{2n-1}+x)\log Q}
		$$
		and
		$$
		g_n(x)=\frac{(L^2+L^4+\dots+L^{2n})\log M +(L+L^3+\dots+L^{2n-1}+x)\log N}{(L+L^2+\dots+L^{2n}+x)\log Q}.
		$$
Since $M\ge N$, $f_n$ is increasing and $g_n$ is decreasing.	

	By solving the following equations, 		
\begin{align*}
			\begin{cases}
				f_n(x)=g_n(y) \\
				\frac{1}{\theta}(\frac{L(1-L^{2n-1})}{1-L}+x)= \frac{L(1-L^{2n})}{1-L}+y,
			\end{cases}
		\end{align*}
we obtain that
\begin{equation}\label{def_xyn}
\left\{\begin{array}{l}
				x_n=\frac{L^{2n+2}-\frac{1}{\theta}L^{2n}+(\frac{1}{\theta}-1)L^2}{\frac{1}{\theta}(L^2-1)}; \\
				y_n=\frac{L^{2n+1}-L}{L^2-1} (\frac{1}{\theta}-1).
\end{array}\right.
		\end{equation}
Note that
\begin{equation}\label{fn=gn}
f_n(x_n)=g_n(y_n).
\end{equation}

For  $\theta\in (\frac{1}{L^2}, 1)$,  there exists an integer $N_0>0$ such that
$$
\theta>\frac{L^{2n}-1}{L^{2n+2}-L^3+L-1},
$$
and
$$
1< x_n\le L^{2n}, \qquad 0\le y_n< L^{2n+1}-L^2,
$$
for all integers $n>N_0$.

 For each integer $k>0$ and $\mathbf{u}\in\Sigma^k$, we have
$$
|J_\mathbf{u}|=\frac{1}{Q}|J_{\mathbf{u}^*}|,
$$
and  it follows that
$$
|J_\mathbf{u}|^\frac{1}{\theta}=\Big(\frac{1}{Q}|J_{\mathbf{u}^*}|\Big)^\frac{1}{\theta}>\Big({\frac{1}{Q}}\Big)^{L^2}|J_{\mathbf{u}^*}|^\frac{1}{\theta},
$$
which is equivalent to
\begin{equation}\label{lk-lk-1}
l(k,\theta)-l(k-1, \theta)\le L^2.
\end{equation}

By \eqref{def_uvn} and \eqref{def_xyn}, it follows that
		\begin{eqnarray*}
			\frac{1}{\theta}([x_n]-1+|\mathbf{u}_n|)&=&\frac{1}{\theta}(|\mathbf{u}_n|+x_n+[x_n]-x_n-1)  \\
			&=&y_n+|\mathbf{v}_n|+\frac{1}{\theta}([x_n]-x_n-1),
		\end{eqnarray*}
and
		\begin{eqnarray*}
[\frac{1}{\theta}([x_n]-1+|\mathbf{u}_n|)]+1&=&|\mathbf{v}_n|+[y_n+\frac{1}{\theta}([x_n]-x_n-1)]+1  \\
			&\le& |\mathbf{v}_n|+[y_n].
		\end{eqnarray*}
Since $f_n$ and $g_n$ are monotone functions, by \eqref{s_mdef} and \eqref{fn=gn},  we have
\begin{eqnarray*}
		s_{[x_n]-1+|\mathbf{u}_n|}&=&f_n([x_n]-1)     \\
        &\le&f_n(x_n) \\
		&=&g_n(y_n) \\
        &\le&g_n([y_n]) \\
		&=&s_{|\mathbf{v}_n|+[y_n]} \\
		&\le& s_{[\frac{1}{\theta}([x_n]-1+|\mathbf{u}_n|)]+1}.
	\end{eqnarray*}
Since $1< x_n\le L^{2n}$ and $0\le y_n< L^{2n+1}-L^2$, it follows that
$$
\min_{|\mathbf{u}_n|+[x_n]-1\leq m\leq l(|\mathbf{u}_n|+[x_n]-1,\theta)} s_m =s_{[x_n]-1+|\mathbf{u}_n|}.
$$
For each $\mathbf{w}_n\in\Sigma^*$ satisfies $|\mathbf{u}_n|\le|\mathbf{w}_n|\le|\mathbf{u}_n|+[x_n]-1$, it follows that
\begin{eqnarray*}
\min_{|\mathbf{w}_n|\leq m\leq l(|\mathbf{w}_n|,\theta)} s_m&\le& s_{|\mathbf{w}_n|}  \\
&\le& s_{|\mathbf{u}_n|+[x_n]-1}  \\
&=&\min_{|\mathbf{u}_n|+[x_n]-1\leq m\leq l(|\mathbf{u}_n|+[x_n]-1|,\theta)} s_m.
\end{eqnarray*}

Similarly, we have that
		\begin{eqnarray*}
			\frac{1}{\theta}([x_n]+1+|\mathbf{u}_n|)&=&y_n+|\mathbf{v}_n|+\frac{1}{\theta}([x_n]-x_n+1),         \\
 \Big[\frac{1}{\theta}([x_n]+1+|\mathbf{u}_n|)\Big]+1 &\ge&|\mathbf{v}_n|+[y_n]+1,
		\end{eqnarray*}
and this implies that
		\begin{eqnarray*}
		s_{[x_n]+1+|\mathbf{u}_n|}&=&f_n([x_n]+1)    \\
        &\ge&f_n(x_n) \\
		&=&g_n(y_n) \\
		&\ge&g_n([y_n]+1)   \\
        &=&s_{|\mathbf{v}_n|+[y_n]+1} \\
		&\ge& s_{[\frac{1}{\theta}([x_n]+1+|\mathbf{u}_n|)]+1}.
	\end{eqnarray*}
Since $1< x_n\le L^{2n}$ and $0\le y_n< L^{2n+1}-L^2$, it follows that
$$
\min_{|\mathbf{u}_n|+[x_n]+1\leq m\leq l(|\mathbf{u}_n|+[x_n]+1,\theta)} s_m =s_{l(|\mathbf{u}_n|+[x_n]+1,\theta)}.
$$
For each $\mathbf{w}_n\in\Sigma^*$ satisfying $|\mathbf{u}_n|+[x_n]+1\le|\mathbf{w}_n|<|\mathbf{u}_{n+1}|$, by \eqref{lk-lk-1},  we have that 
\begin{eqnarray*}
\min_{|\mathbf{w}_n|\leq m\leq l(|\mathbf{w}_n|,\theta)} s_m&\le& \min\{s_{l(|\mathbf{w}_n|,\theta)}, s_{|\mathbf{u}_{n+1}|}\}  \\
&\le& s_{l(|\mathbf{u}_n|+[x_n]+1,\theta)}  \\
&=&\min_{|\mathbf{u}_n|+[x_n]+1\leq m\leq l(|\mathbf{u}_n|+[x_n]+1|,\theta)} s_m.
\end{eqnarray*}		

Hence, for every $n>N_0$, For $\mathbf{z}_n$ satisfying $|\mathbf{z}_n|=|\mathbf{u}_n|+[x_n]$, we have that
$$
\max_{|\mathbf{u}_n|\le k\le |\mathbf{u}_{n+1}|}\min_{k\leq m\leq l(k,\theta)} s_m =\max_{|\mathbf{z}_{n}|-1\le k\le |\mathbf{z}_{n}|+1}\min_{k\leq m\leq l(k,\theta)} s_m,
$$
and we complete the proof of the claim.

The claim implies that
\begin{equation}\label{sm=min}
		\limsup_{k\to\infty} \min_{k\leq m\leq l(k,\theta)} s_m=\limsup_{n\to\infty} \max_{|\mathbf{z}_{n}|-1\le k\le |\mathbf{z}_{n}|+1}\min_{k\leq m\leq l(k,\theta)} s_m.
\end{equation}
Therefore by \eqref{lk-lk-1}, \eqref{sm=min} and Lemma \ref{continue},  we have 		
\begin{eqnarray*}
			\limsup_{k\to\infty} \min_{k\leq m\leq l(k,\theta)} s_m&=&\limsup_{n\to\infty} \max_{|\mathbf{z}_{n}|-1\le k\le |\mathbf{z}_{n}|+1}\min_{k\leq m\leq l(k,\theta)} s_m  \\
&=&\lim_{n\to\infty} f_n(x_n)   \\
			&=&\lim_{n\to \infty}\frac{(\theta L+\frac{L^2-1}{L(L^{2n}-1)}x_n)\log M+\log N}{(\theta L+1+\frac{L^2-1}{L(L^{2n}-1)}x_n)\log Q}   \\
			&=&\frac{L\theta \log M  +\log N}{L \theta \log Q +\log Q}  \\
	        &=&f(\theta),
		\end{eqnarray*}
where $\theta\in(\frac{1}{L^2},1)$.

Since
		$$
		\ubd E=\frac{L\log M +\log N}{L\log Q  +\log Q},  \qquad  \hdd E=\lbd E=\frac{\log M +L\log N}{\log Q  +L\log Q},
		$$
		by the continuity of the intermediate dimension, we have
$$
\uid E=f(\theta),
$$
for $\theta\in[\frac{1}{L^2},1]$ and $\uid E=\hdd E$ for $\theta\in[0,\frac{1}{L^2}]$.  Moreover, by Corollary \ref{cor_HMS}, we have that
$$
\lid E=\hdd E=\frac{\log M +L\log N}{\log Q  +L\log Q},
$$ for $\theta\in[0,1]$.
	\end{proof}

	Next, we give some examples to explain our main conclusions and show some interesting facts.  The first example shows that the  upper intermediate and lower intermediate dimensions are different even for homogeneous  Moran sets.	
\begin{exmp}\label{exm1}
Given $J=[0,1]$, $c_k =\frac{1}{4}$ and		
$$
		n_k=\left\{ \begin{array}{lc}
			3, & \qquad (2n!)^2< k\le((2n+1)!)^2,\\
			2, &  		((2n+1)!)^2< k\le((2n+2)!)^2,
		\end{array}\right.
		$$
for every integer $k>0$. Let $E$ is the corresponding homogeneous Moran set given by \eqref{attractor}.  Then
\begin{eqnarray*}
		\uid E &=&\left\{ \begin{array}{ll}
                       \ubd E=\frac{\log3}{2\log2}, & \textit{  for } \theta\in (0, 1], \\
                       \hdd E, & \textit{  for } \theta=0;
                       \end{array}\right.    \\
		\lid E&=&\hdd E ={\frac{1}{2}},\qquad \quad \textit{ for } \theta\in [0, 1].
\end{eqnarray*}

	\end{exmp}
	\begin{proof}
		For each $k$, by the definition of $s_k$, there exist two integers $m(k)$ and $n(k)$ such that
		$$
		s_k=\frac{m(k)\log 2 +n(k)\log 3}{2m(k)\log 2 +2n(k)\log 2}.
		$$
		Since for all integers $m>0, n>0$, the following inequality holds
		$$
		{\frac{m\log 2 +n\log 3}{2m\log 2 +2n\log 2}}<{\frac{m\log 2 +(n+1)\log 3}{2m\log 2 +2(n+1)\log 2}} ,
		$$
		we obtain that
		$$
		\frac{1}{2}\le s_k\le\frac{\log 3}{2\log 2}.
		$$
		By \eqref{def_s***}, this implies that
		$$
		\frac{1}{2}\le \hdd E\le\ubd E\le\frac{\log 3}{2\log 2}.
		$$
Therefore, it is sufficiently to show that  $	\uid E\geq\frac{\log3}{2\log2}$ and $
		\lid E\leq {\frac{1}{2}}$ for $\theta\in(0,1]$.

		Let
		$$
		a_n={\frac{1}{4^{n((n-1)!)^2}}}, \qquad b_n={\frac{1}{4^{(n!)^2}}}.
		$$
		For every $n>2$, there exist $\mathbf{u}_n\in\Sigma^*$ and  $\mathbf{v}_n\in\Sigma^*$ such that
		$$
		a_n=\lvert J_{\mathbf{u}_n} \rvert,  \qquad b_n=\lvert J_{\mathbf{v}_n} \rvert,
		$$
and		it is clear that $s_k$ is monotonically increasing if  $|\mathbf{v}_{2n}|<k\le |\mathbf{v}_{2n+1}|$ and monotonically decreasing if $|\mathbf{v}_{2n+1}|<k\le |\mathbf{v}_{2n+2}|$.
		
Fix $\theta\in(0,1]$,  there exists $N$ such that  $\frac{1}{N} \le \theta <\frac{1}{N-1}$, it is follows that
		$$
		b_{N+i} \le a_{N+i}^N \le a_{N+i}^{{\frac{1}{\theta}}} < a_{N+i}< b_{N+i-1}.
		$$
Since $\lvert \mathbf{u}_n \rvert =n((n-1)!)^2,$,  there exist two integers $c_1(k)$ and $c_2(k)$ such that
		$$
		s_{\lvert \mathbf{u}_{2k+1} \rvert}={\frac{c_1(k) \log3 +c_2(k) \log2}{2c_1(k) \log2 +2c_2(k) \log2}}  ,
		$$
where $c_1(k)>(2k)((2k)!)^2$ and  $c_2(k)<((2k)!)^2$.  Letting $\delta=a_{2k+1}$, we have that  $s_{\delta, \theta}=s_{(2k+1)((2k)!)^2}$, and this implies that
		\begin{eqnarray*}
\uid E 	&\ge&\lim_{k\to \infty} s_{(2k+1)((2k)!)^2}\\
			&\ge&\limsup_{k\to \infty} {\frac{c_1(k) \log3 +c_2(k) \log2}{2c_1(k) \log2 +2c_2(k) \log2}} \\
			&\geq&{\frac{\log 3}{2\log 2}}.
		\end{eqnarray*}

Similarly, there exist two integers $c_3(k)$ and $c_4(k)$ such that
		$$
		s_{\lvert \mathbf{u}_{2k+2} \rvert}={\frac{c_3(k) \log2+c_4(k) \log3}{2 c_3(k) \log2 +2 c_4(k) \log2}},
		$$
where $c_3(k)>(2k+1)((2k+1)!)^2$ and $c_4(k)<((2k+1)!)^2$,
and   letting $\delta=a_{2k+2}$,  we have that  $s_{\delta, \theta}=s_{(2k+2)((2k+1)!)^2}$,  and similarly, it implies that
$$
 \lid E  \le {\frac{1}{2}}.
$$
	\end{proof}

	In the next example, we construct two homogeneous Moran sets,  and the upper intermediate dimension of their product is strictly less than
	the sum of upper intermediate dimensions.
	\begin{exmp}\label{exm2}
		Let $E$ be the homogeneous Moran set in Example \ref{exm1}. For each integer $k>0$, let $c_k =\frac{1}{4}$ and
		$$
		l_k=\left\{ \begin{array}{lc}
			2, & \qquad (2n!)^2< k\le((2n+1)!)^2,\\
			3, &  		((2n+1)!)^2< k\le((2n+2)!)^2.
		\end{array}\right.
		$$
Let $F$ be the  corresponding Moran set  given by \eqref{attractor} with respect to $\{c_k\}$ and $\{l_k\}$. Then
		$$
		\uid (E\times F)<\uid E +\uid F,
		$$
		with $\theta \in (0,1]$.
	\end{exmp}
	\begin{proof}
Since $\uid E=\uid F ={\frac{\log3}{2\log2}}$ for $\theta\in(0,1]$, it is sufficient to prove that 
		$$
		\ubd (E\times F) < {\frac{\log3}{\log2}}.
		$$
	By considering  the cover of $E\times F$ with squares with length of ${\frac{1}{4^k}}$, we have
		$$
		N_{{\frac{1}{4^k}}} (E\times F)\le 2^k *3^k =6^k,
		$$
and this implies that 		
		$$
		\ubd (E\times F) \le \limsup_{k\to \infty} {\frac{\log{6^k}}{\log{4^k}}} <{\frac{\log3}{\log2}}=\uid E +\uid F,
		$$
		when $\theta \in (0,1]$.
	\end{proof}
	
	In the next example, we construct a Moran set with  $c_*=0$,  and all the dimensions are identical.
	\begin{exmp}\label{exm3}
Let $E$ be a homogeneous Moran set with $n_k =2^k$ and $c_k={\frac{1}{3^{k+1}}}$. Then the intermediate dimension of $E$ exists, and
		$$
		\dim_{\theta} E=\hdd E=\bod E=\frac{\log 2}{\log 3} .
		$$
	\end{exmp}
	\begin{proof}
		Since $\underline{c}_k=c_k$ and $M_k=c_1c_2\ldots c_k$, we have that
		$$
	\lim_{k\to +\infty} \frac{\log{\underline{c}_k}}{\log{M_k}}=	\lim_{k\to \infty}\frac{\log c_k}{\log c_1 \dots c_k}=0.
		$$
By  \eqref{def_hmssk}, it is clear that
	$$
		s_k=\frac{(k+1)\log 2}{(k+3)\log 3} \le \frac{(k+2)\log 2}{(k+4)\log 3}=s_{k+1}.
		$$
This  implies that $s^*=s_*=\frac{\log 2}{\log 3}$,  and the 	 intermediate dimension of $E$ exists and
$$
\dim_{\theta} E=\hdd E=\bod E=\frac{\log 2}{\log 3}.
		$$
		
	\end{proof}

Finally, we give an example for the visualization of intermediate dimension by applying the method used in the proof of Proposition \ref{Prop_MF}.
\begin{exmp}
Given $J=[0,1]$ and a real $r\in(0,\frac{1}{2})$. Let $N>1,M>1,L>1$ be integers satisfying that $N\le M < \frac{1}{r}$. For each integer $k>0$, let $c_k =r$, and
		\begin{align*}
			n_k=
			\begin{cases}
				N,\quad \textit { if }k=1 \textit { or  } L+L^2+\dots+L^{2n-2}<k\le L+L^2+\dots+L^{2n-1},    \\
				M,\quad \textit { if } L+L^2+\dots+L^{2n-1}<k\le L+L^2+\dots+L^{2n}.
			\end{cases}
		\end{align*}
Let  $E$ be the corresponding Homogenous Moran set.	Then we have
$$
\uid E=\frac{L\log M+\frac{1}{\theta}\log N}{-(L+\frac{1}{\theta})\log r}
$$
for $\theta \in (\frac{1}{L^2}, 1]$, and $\uid E=\hdd E=\frac{L\log N+\log M}{-(L+1)\log r}$ for $\theta \in [0,\frac{1}{L^2}].$

	\end{exmp}

\end{document}